\documentclass{amsart}

\newtheorem{theorem}{Theorem}[section]
\newtheorem{lemma}[theorem]{Lemma}

\theoremstyle{definition}
\newtheorem{definition}[theorem]{Definition}

\theoremstyle{remark}

\numberwithin{equation}{section}

\begin{document}

\title{MAXIMIZING THE FIRST EIGENVALUE OF THE JACOBI OPERATOR}

\author{J. Fabio B. Montenegro}
\address{Department of Mathematics, Universidade Federal do Cear\'a, Fortaleza, 60175-020, Cear\'a}
\email{fabio@mat.ufc.br}

\author{F. Damiana Vieira}
\address{Department of Mathematics, Universidade Federal do Cariri, Brejo Santo 63260-000, Cear\'a}
\email{damiana.vieira@ufca.edu.br}
%\thanks{The work was done with the support of the Coordenação de Aperfeiçoamento de Pessoa de Nível Superior-Brasil (CAPES) - Financing Code 001.}

%    General info

%\date{January 1, 2001 and, in revised form, June 22, 2001.}


\begin{abstract}
We consider the Jacobi operator,
defined on a closed oriented hypersurfaces immersed in the Euclidean space with the same volume of the unit sphere. We show a local generalization for the classical result of the Willmore functional for the Euclidean sphere. As a consequence, we prove that the first eigenvalue of the Jacobi operator in the Euclidean sphere is a local maximum and this result is a global one in the closed oriented surfaces space of $\mathbb{R}^3$ and genus zero.
\end{abstract}

\maketitle

\section*{Introduction}

Let $M^n$ be a closed oriented hypersurface immersed in $\mathbb{R}^{n+1}$. Consider the differential operator
\begin{equation}
\label{Oprator1}
L=-\Delta-|II|^2, 
\end{equation}
where $\Delta$ is the Laplace-Beltrami operator and $|II|^2=\sum_{j=1}^{n}k_j^2$ is the squared norm of the second fundamental form of $M$. This operator arise naturally  in the study of stability of geometric problems, as known as Jacobi operator, observed in the works of Barbosa and do Carmo \cite{B},  Barbosa, do Carmo and Eschenburg \cite{C}, by appearing in the calculation to the second variation of the volume at $M$. It is also found in the works of Harrel II \cite{E}, Papanicolaou \cite{G} and Harrell II and Loss \cite{F} that deal with partial and total proof of the Alikakos-Fusco conjecture proposed in \cite{A}. This conjecture is related to a physical problem involving stability of interfacial surfaces, in which the authors realized that the instability of these surfaces was associated with the negative eigenvalues of the Laplace-Beltrami operator $L$. Since the first eigenvalue of $L$ is always negative, all attention was restricted on the second eigenvalue, where stability and characterization theorems were obtained, as it is presented by Harrell and Loss in \cite{F}, where Alikakos-Fusco conjecture follows as a particular case. In order to present this result, we are denoting by $\mathbf{H}=\sum_{j=1}^{n}k_j$ the mean curvature non-normalized of $M$.

\begin{theorem}[\cite{F}] 
Let $\Omega$ be a smooth compact oriented hypersurface of dimension $n$ immersed in $\mathbb{R}^{n+1};$ in particular self-intersections are allowed. The metric on that surface is the standard Euclidean metric inherited from $\mathbb{R}^{n+1}$. Then the second eigenvalue $\lambda_2$ of the operator
\begin{equation}
\label{Oprator2}
\mathcal{L}=-\Delta-\displaystyle\frac{1}{n}\; \mathbf{H}^2
\end{equation}
is strictly negative unless $\Omega$ is a sphere, in which case $\lambda_2$ equals zero.
\end{theorem}

We can look at this theorem as a characterization of a round sphere. So we decided to investigate if this also happens for the first eigenvalue. That is, among all the orientable compact hypersurfaces immersed in Euclidean space, with the same volume as the unit sphere, which one has the highest first eigenvalue?
It is reasonable to expect such a hypersurface be a sphere, but we only get to prove it in the case of surface, see Theorem \ref{FirstTheorem}, and we get a partial result for higher dimensions, see Theorem \ref{SecondTheorem}.
It is important to note that the results that we have obtained regarding the operator $\mathcal{L}$, defined in (\ref{Oprator2}), also apply to Jacobi operator $L$ in (\ref{Oprator1}), since $(1/n)\mathbf{H}^2\leq |II|^2$, with equality occurring in the Euclidean sphere.

In development of the results, we use the Willmore energy functional associated to the closed orientable, smooth surface $S\subset\mathbb{R}^3$, $W:\mathcal{F} \longrightarrow \mathbb{R}$ defined by	
$$
W(f)=\frac{1}{2\pi}\displaystyle \int_{f(S)}\!\! \mathrm{H^2}\, dS, \label{FW}
$$
where  $\mathcal{F}$ is the space of all $\mathcal{C}^{\infty}$-embedding $f:S\rightarrow \mathbb{R}^3$, $H=(k_1+k_2)/2$ is the mean curvature of $f(S)$ consider as a hypersurface of $\mathbb{R}^3$. Here $k_1$,  $k_2$ are the two classical principal curvatures of the surface $f(S)$ and $dS$ is the area element of the induced metric on $S$.

The functional $W$ was first studied by W. Blaschke \cite{Bl} and G. Thomsen \cite{T}, who established the most important property of $W$: The functional $W$ is invariant under conformal changes of metric of the ambient space $\mathbb{R}^3$. 
It appears naturally in other areas of knowledge, such as the study of: elastic shells \cite{D}, \cite{I} and cell membranes \cite{J}. Among mathematics it became well known by the famous Willmore Conjecture, proposed in \cite{L} and proved by Cod\'a and Neves in \cite {H}.

In \cite{L}, Willmore proved that among all embedded surfaces in $\mathbb{R}^3$, the energy functional attains its minimum in the Euclidean sphere. As a consequence, we prove that the first eigenvalue of the $\mathcal{L}$ operator attains its global maximum in the sphere $\mathbb{S}^2$.
Also, for $n>2$ we will proceed similar to Papanicolaou in \cite{G}, by giving local results in the sense that $M$ is a sufficiently small perturbation of the unity sphere $\mathbb{S}^n$, with the same volume of the unity sphere. So we prove that Willmore functional attains, at sphere $\mathbb{S}^n$, a local minimum and, consequently, the first eigenvalue of the operator $\mathcal{L}$ attains, at sphere $\mathbb{S}^n$, a local maximum.

\section{Variations for hypersurfaces of constant volume}\label{Var}

Let $\xi:M \rightarrow \mathbb{R}^{n+1}$ be an immersion of an orientable, n-dimensional differentiable manifold into $\mathbb{R}^{n+1}$. A variation of $\xi$ is a differentiable application
	\begin{equation*}
	X:(-\varepsilon,\varepsilon)\times M\rightarrow \mathbb{R}^{n+1}
	\end{equation*} 
such that, for all $t\in(-\varepsilon,\varepsilon)$ and  $p\in M$ , the map $X_t:M^n\rightarrow \mathbb{R}^{n+1}$ defined by $X_t(p)=X(t,p)$, is an isometric immersion, with $X_0=\xi .$

Let  $vol(M)$ be the n-volume of $M$ in the induced metric of $\mathbb{R}^{n+1}$. We will say that a variation is volume-preserving  if $vol(M_t)=vol(M)$, for all $t\in (-\varepsilon,\varepsilon)$, where $M_t=X_t(M)$.

For each  $X_t$ immersion, the coefficient of the metric tensor corresponding will be denoted by $\mathbf{g}_{ij}=\mathbf{g}_{ij}(t)$ and $g_{ij}=\mathbf{g}_{ij}(0)$, by $\mathbf{g}=\mathrm{det}(\mathbf{g}_{ij})$ where det denotes the determinant, the second fundamental form by $\mathbf{h}_{ij}=\mathbf{h}_{ij}(t)$, and its trace will be $\mathbf{H}$. We will make frequent use of the following lemma.

\begin{lemma}
	Let $f:\mathbb{S}^n\rightarrow\mathbb{R}$ be a smooth function. Then there exists a volume-preserving variation of $Id:\mathbb{S}^n\to\mathbb{S}^n\subset \mathbb{R}^{n+1}$,
$X:(-\varepsilon ,\varepsilon )\times \mathbb{S}^n \to \mathbb{R}^{n+1}$
given by
\begin{equation}
\label{X_t}
X(t,p)=(1+t\, f(p)+\varphi(t))p,
\end{equation}
where $\varphi:(-\varepsilon,\varepsilon)\to(-\delta,\delta)$ is a smooth function obtained by the Implicit Function Theorem, such that $\varphi(0)=0$. Furthermore,  $\varphi'(0)=0$ if $f$ have zero average.
\end{lemma}

\begin{proof}
Consider the family of embedding
$X:(-\varepsilon ,\varepsilon )\times \mathbb{S}^n \to \mathbb{R}^{n+1}$ given by
$$
X(t,s,p)=(1+t\, f(p)+s)p,
$$
for $t$ and $s$ sufficiently small.

Let $\Phi:U\subset\mathbb{R}^n\rightarrow\mathbb{S}^n\subset \mathbb{R}^{n+1}$ be a parametrization of the unit sphere such that $\Phi(U)=\mathbb{S}^n-\{p_0\}$, $p_0\in\mathbb{S}^n$. For simplicity of notation, we write $f(x)$ instead of $f(\Phi (x))$, for $x\in U$. 

\hspace{-0.45cm}The volume of $X_{(t,s)}(\mathbb{S}^n)$ is given by
$$
vol(t,s)=\int_U\sqrt{g(t,s)}\;dx.
$$
Then,
$$
\frac{\partial vol}{\partial s}(t,s)=\int_U\frac{\partial}{\partial s}\left( \sqrt{g(t,s)}\right)dx
$$
and, for $t=s=0$, we have
$$
\frac{\partial vol}{\partial s}(0,0)=n\, \int_U \sqrt{g}\, dx=n\, vol(\mathbb{S}^n)>0.
$$
By the Implicit Function Theorem applied the function $(t,s)\mapsto vol(t,s)$, there exist a neighborhood $(-\varepsilon , \varepsilon )\times (-\delta , \delta )$ of the origin $(0,0)$ and a smooth function $\varphi :(-\varepsilon , \varepsilon )\to
(-\delta , \delta )$ with $\varphi (0)=0$, satisfying
$vol(t,\varphi (t))=vol(0,0),\;\; \forall \, t\in (-\varepsilon , \varepsilon )$
and
\begin{equation}
\label{varphit0}
\varphi '(0)=-\frac{\left(\partial vol/ \partial t\right)(0,0)}{\left(\partial vol / \partial s\right)(0,0)}=
-\frac{1}{ vol(\mathbb{S}^n)}\int_{\mathbb{S}^n}f\mathrm{d}\mathbb{S}^n.
\end{equation}
Therefore, the variation
$$
X(t,p)=(1+t\, f(p)+\varphi(t))p
$$
is volume-preserving and $\varphi'(0)=0$ if $\displaystyle\int_{\mathbb{S}^n}f=0$.
\end{proof}
The following lemmas will be essential for the proof of some theorems of this paper.
  
\begin{lemma}
\label{lemma2}
For each $t\in(-\varepsilon,\varepsilon)$ we associate the immersion $X_t$ definite in (\ref{X_t}), with $f$ having zero average, and consequently a embedded hypersurface $M_t$ at $\mathbb{R}^{n+1}$, with the same volume of $M_0=\mathbb{S}^n.$ For each hypersurface we denote by $\mathbf{H}=\mathbf{H}(t)$ the mean curvature non-normalized associated a inherited metric $\mathbf{g}_{ij}=\mathbf{g}_{ij}(t)$. Then 
\begin{equation}
\label{sqrtGt0}
\frac{\partial\sqrt{\mathbf{g}}}{\partial t}\Big|_{t=0}=nf\sqrt{{g}}
\end{equation}
\begin{equation}
\label{Ht0}
\frac{\partial \mathbf{H}}{\partial t}\Big|_{t=0}=-nf-\Delta f 
\end{equation}
\end{lemma}
\begin{proof}
Let $X(t,x)=(1+t\, f(x)+\varphi(t))\Phi(x)$ be a parametrization of the $M_t$, where  $\Phi:U\subset\mathbb{R}^n\rightarrow\mathbb{S}^n\subset \mathbb{R}^{n+1}$ be a parametrization of the unit sphere. Then
$$
\frac{\partial X}{\partial x_i}=t\frac{\partial f}{\partial x_i}\Phi +(1+tf+\varphi(t))\frac{\partial \Phi}{\partial x_i}
$$
where we use the simplified notation $f=f\circ \Phi$. One easily checks that
\begin{equation}
\label{gij}
\mathbf{g}_{ij}= t^2\frac{\partial f}{\partial x_i}\frac{\partial f}{\partial x_j}+{(1+tf+\varphi(t))}^{2}{g}_{ij}
\end{equation}
Notice that $\mathrm{det}(\mathbf{g}_{ij})=\mathrm{det}(g_{ij})+ O(t)$. Hence, for $t$ sufficiently small $\mathbf{g}_{ij}(t)$ is a Riemannian metric.

Deriving the metric coefficients given in (\ref{gij}) we have
\begin{equation}
\label{gijt}
\frac{\partial \mathbf{g}_{ij}}{\partial t}
=2t\frac{\partial f}{\partial x_i}\frac{\partial f}{\partial x_j}\\
+2(1+tf+\varphi(t) )(f+\varphi'(t)){g}_{ij}
\end{equation}
and
\begin{equation}
\label{gijt0}
\frac{\partial \mathbf{g}_{ij}}{\partial t}\Big|_{t=0}=2f{g}_{ij}.
\end{equation}

We will also need to derive the inverse metric coefficients.
To this end we use that 
$$
\sum_{j=1}^{n}\mathbf{g}^{ij}\mathbf{g}_{jk}={\delta^{i}_{k}}.
$$
Then
$$
\sum_{j=1}^{n}\frac{\partial \mathbf{g}^{ij}}{\partial t}\mathbf{g}_{jk}= -
\sum_{j=1}^{n}\mathbf{g}^{ij}\frac{\partial \mathbf{g}_{jk}}{\partial t}.
$$
that is, by (\ref{gijt})
\begin{equation}
\label{ijgt}
\frac{\partial \mathbf{g}^{ij}}{\partial t}
=\bigg(\!\!-2t\!\sum_{p,r=1}^{n} \frac{\partial f}{\partial x_p}\frac{\partial f}{\partial x_r}-2(1+tf+\varphi(t) )(f+\varphi'(t))g_{pr}\bigg)\mathbf{g}^{ip}\mathbf{g}^{rj}
\end{equation}
and
\begin{equation}
\label{ijgt0}
\frac{\partial \mathbf{g}^{ij}}{\partial t}\Big|_{t=0}=-2f{g^{ij}}
\end{equation}

We are now able to derive the square root of the metric determinant. We have 
\begin{equation}
\label{sqrtGt}
\frac{\partial}{\partial t}\sqrt{\mathbf{g}}=\frac{1}{2}\bigg(\displaystyle\sum_{i,j=1}^{n}\mathbf{g}^{ij}\;\frac{\partial \mathbf{g}_{ij}}{\partial t}\bigg)\sqrt{\mathbf{g}}.
\end{equation}
Then by (\ref{gijt0})
$$
\frac{\partial\sqrt{\mathbf{g}}}{\partial t}\Big|_{t=0}=
n\, f\sqrt{g}.
$$

I remains to prove (\ref{Ht0}) and to this end we note that
\begin{equation}
\label{N.Xxi}
\bigg\langle \mathbf{N},\frac{\partial X }{\partial x_i}\bigg\rangle =
\bigg\langle \frac{\partial \mathbf{N} }{\partial t}, \mathbf{N}\bigg\rangle=0
\end{equation}
Then,
$$
\bigg\langle \frac{\partial \mathbf{N}}{\partial t},\frac{\partial X }{\partial x_i}\bigg\rangle = 
-\bigg\langle \mathbf{N},\frac{\partial^2 X }{\partial t\partial x_i}\bigg\rangle
$$
and
\begin{equation}
\label{Nt}
\frac{\partial\mathbf{N}}{\partial t}=
-\langle \mathbf{N},\Phi\rangle\sum_{i,j=1}^n\mathbf{g}^{ij}\frac{\partial f}{\partial x_i}\frac{\partial X}{\partial x_j}-(f+\varphi'(t))\sum_{i,j=1}^n\mathbf{g}^{ij}\bigg\langle\mathbf{N},\frac{\partial \Phi}{\partial x_i}\bigg\rangle\frac{\partial X}{\partial x_j}
\end{equation}
For $t=0$,
\begin{equation}
\label{Nt0}
\frac{\partial \mathbf{N}}{\partial t}\Big|_{t=0}=-\mathrm{grad}f
\end{equation}

We also need to derive the coefficients of the second fundamental form.
By definition $\mathbf{h}_{ij}=\langle \partial X /\partial x_i,\partial N/\partial x_j\rangle$. Thus, by (\ref{N.Xxi}), we have
$$
\mathbf{h}_{ij}=-\bigg\langle \frac{\partial^2 X}{\partial x_i\partial x_j} ,\mathbf{N}\bigg\rangle .
$$
So,
\begin{equation}
\label{hijt}
\frac{\partial \mathbf{h}_{ij}}{\partial t}=
-\bigg\langle \frac{\partial^3 X}{\partial t\partial x_i\partial x_j} ,\mathbf{N} \bigg\rangle
-\bigg\langle \frac{\partial^2 X}{\partial x_i\partial x_j} ,\frac{\partial \mathbf{N}}{\partial t}\bigg\rangle
\end{equation}
For $t=0$,
$$
h_{ij}=\mathbf{h}_{ij}(0)=-\bigg\langle \frac{\partial^2 \Phi}{\partial x_i\partial x_j} ,\Phi\bigg\rangle=g_{ij}
$$
and
\begin{equation}
\label{hijt0}
\frac{\partial \mathbf{h}_{ij}}{\partial t}\Big|_{t=0}=fg_{ij}-\frac{\partial^2f}{\partial x_i\partial x_j}
+\sum_{k=1}^n\Gamma_{ij}^k\frac{\partial f}{\partial x_k}
\end{equation}
where
$$
\Gamma_{ij}^k
=\sum_{l=1}^ng^{kl}\bigg\langle\frac{\partial^2 \Phi}{\partial x_i\partial x_j},\frac{\partial \Phi}{\partial x_l}\bigg\rangle 
$$ 
are the Christoffel symbols.

Finally, as $\mathbf{H}=\displaystyle\sum_{i,j=1}^{n}\mathbf{g}^{ij}\mathbf{h}_{ij}$ we have
\begin{equation}
\label{Ht}
\frac{\partial\mathbf{H}}{\partial t}=\sum_{i,j=1}^{n}\frac{\partial\mathbf{g}^{ij}}{\partial t}\mathbf{h}_{ij}
+\sum_{i,j=1}^{n}\mathbf{g}^{ij}\frac{\partial\mathbf{h}_{ij}}{\partial t}
\end{equation}
and, by (\ref{ijgt0}) and (\ref{hijt0})
$$
\frac{\partial\mathbf{H}}{\partial t}\Big|_{t=0}=-nf-\Delta f.
$$
Here we use that
$$
\Delta f=\sum_{i,j=1}^ng^{ij}\frac{\partial^2f}{\partial x_i\partial x_j}
-\sum_{i,j,k=1}^ng^{ij}\Gamma_{ij}^k\frac{\partial f}{\partial x_k}\;.
$$
\end{proof}

\begin{lemma}	
\label{lemma3}
In the same hypotheses as the previous lemma we have
\begin{equation}
\label{sqrtG2t0}			
\frac{\partial^2\sqrt{\mathbf{g}}}{\partial t^2}\Big|_{t=0}=\bigg(|\mathrm{grad} f|^2+(n^2-n)f^2+n\varphi''(0)\bigg)\sqrt{{g}}
\end{equation}
\begin{equation}
\label{H2t0}		
\frac{\partial^2\mathbf{H}}{\partial t^2}\Big|_{t=0}=
(2-n)|\mathrm{grad}f|^2+2nf^2-n\varphi''(0)+4f\Delta f
\end{equation}
\end{lemma}
\begin{proof}

Deriving the expression in (\ref{gijt}) we have
\begin{equation}
\label{gij2t}
\frac{\partial^2\mathbf{g}_{ij}}{\partial t^2}=2\frac{\partial f}{\partial x_i}\frac{\partial f}{\partial x_j}+2(f+\varphi'(t) )^2g_{ij}
+2(1+tf+\varphi(t) )\varphi''(t) g_{ij}
\end{equation}
For $t=0$
\begin{equation}
\label{gij2t0}
\frac{\partial^2\mathbf{g}_{ij}}{\partial t^2}\Big|_{t=0}=2\,\frac{\partial f}{\partial x_i}\frac{\partial f}{\partial x_j}
+\Big(2f^2+2\varphi''(0)\Big)g_{ij}
\end{equation}

Deriving the expression in (\ref{ijgt}) we have
\begin{equation}
\label{ijg2t}
\frac{\partial^2\mathbf{g}^{ij}}{\partial t^2}=
-2\sum_{r,p=1}^{n}\left( t\frac{\partial f}{\partial x_p}\frac{\partial f}{\partial x_r}
+(1+tf+\varphi(t) )(f+\varphi'(t) )g_{pr}\right)\frac{\partial\mathbf{g}^{ip}}{\partial t}\mathbf{g}^{rj}
\end{equation}
$$
-2\sum_{r,p=1}^{n}\left( t\frac{\partial f}{\partial x_p}\frac{\partial f}{\partial x_r}
+(1+tf+\varphi(t) )(f+\varphi'(t) )g_{pr}\right)\frac{\partial\mathbf{g}^{rj}}{\partial t}\mathbf{g}^{ip}
$$
$$
-2\sum_{r,p=1}^{n}\left( \frac{\partial f}{\partial x_p}\frac{\partial f}{\partial x_r}
+(f+\varphi'(t))^2g_{pr}+(1+tf+\varphi(t) )\varphi''(t) g_{pr}\right)\mathbf{g}^{ip}\mathbf{g}^{rj}
$$
Then by (\ref{ijgt0}),
$$
\frac{\partial^2\mathbf{g}^{ij}}{\partial t^2}\Big|_{t=0}
=-2\sum_{r,p=1}^{n} \frac{\partial f}{\partial x_p}\frac{\partial f}{\partial x_r}g^{ip}g^{rj}
+\left(6f^2-2\varphi''(0)\right)g^{ij}
$$
Now we observe that
$$
\sum_{i,j=1}^n\left(\sum_{r,p=1}^{n} \frac{\partial f}{\partial x_p}\frac{\partial f}{\partial x_r}g^{ip}g^{rj}\right)g_{ij}=|\mathrm{grad}f|^2=
\sum_{i,j=1}^n\left(\frac{1}{n}|\mathrm{grad}f|^2g^{ij}\right)g_{ij}
$$
Then
$$
\sum_{r,p=1}^{n} \frac{\partial f}{\partial x_p}\frac{\partial f}{\partial x_r}g^{ip}g^{rj}=\frac{1}{n}\,|\mathrm{grad}f|^2g^{ij}
$$
and
\begin{equation}
\label{ijg2t0}
\frac{\partial^2\mathbf{g}^{ij}}{\partial t^2}\Big|_{t=0}=
\left(-\frac{2}{n}\,|\mathrm{grad}f|^2+6f^2-2\varphi''(0)\right)g^{ij}
\end{equation}

Deriving the expression in (\ref{sqrtGt}) we have
\begin{equation}
\label{sqrtG2t}
\frac{\partial^2\sqrt{\mathbf{g}}}{\partial t^2}
=\frac{1}{2}\bigg(\sum_{i,j=1}^{n}\frac{\partial \mathbf{g}^{ij}}{\partial t}\;
\frac{\partial\mathbf{g}_{ij}}{\partial t}\bigg)\sqrt{\mathbf{g}}
\end{equation}
$$
+\frac{1}{2}\bigg(\sum_{i,j=1}^{n}\mathbf{g}^{ij}\;
\frac{\partial^2\mathbf{g}_{ij}}{\partial t^2}\bigg)\sqrt{\mathbf{g}}
+\frac{1}{2}\bigg(\sum_{i,j=1}^{n}\mathbf{g}^{ij}\;
\frac{\partial\mathbf{g}_{ij}}{\partial t}\bigg)\frac{\partial\sqrt{\mathbf{g}}}{\partial t}	
$$
Then by (\ref{gijt0}), (\ref{ijgt0}), (\ref{gij2t0}) and (\ref{sqrtGt0})
$$
\frac{\partial^2}{\partial t^2}\sqrt{\mathbf{g}}\, \Big|_{t=0}
=\left( |\mathrm{grad}f|^2+(n^2-n)f^2+n\varphi''(0)\right)\sqrt{g}
$$

Now we observe that 
$$
0=\frac{\partial^2}{\partial t^2}\left\langle N,N\right\rangle=
2\bigg\langle \frac{\partial^2 N}{\partial^2 t},N\bigg\rangle
+2\Big| \frac{\partial N}{\partial t}\Big|^2
$$
Then
$$
\bigg\langle \frac{\partial^2 N}{\partial^2 t}\Big|_{t=0},\Phi\bigg\rangle
=-\Big| \frac{\partial N}{\partial t}\Big|_{t=0}\Big|^2
=-|\mathrm{grad}f|^2
$$
On the other hand, we have that
$$
0=\frac{\partial^2}{\partial t^2}\bigg\langle N,\frac{\partial X}{\partial x_i}\bigg\rangle
=\bigg\langle \frac{\partial^2N}{\partial t^2},\frac{\partial X}{\partial x_i}\bigg\rangle
+2\bigg\langle \frac{\partial N}{\partial t},\frac{\partial^2 X}{\partial t\partial x_i}\bigg\rangle
+\bigg\langle N,\frac{\partial^3 X}{\partial t^2\partial x_i}\bigg\rangle
$$
that is,
$$
\bigg\langle \frac{\partial^2N}{\partial t^2}\Big|_{t=0},\frac{\partial \Phi}{\partial x_i}\bigg\rangle
=-2\bigg\langle \frac{\partial N}{\partial t}\Big|_{t=0},\frac{\partial^2 X}{\partial t\partial x_i}\Big|_{t=0}\bigg\rangle
-\bigg\langle \Phi,\frac{\partial^3 X}{\partial t^2\partial x_i}\Big|_{t=0}\bigg\rangle
$$
By (\ref{Nt}) we have
$$
\bigg\langle \frac{\partial^2N}{\partial t^2}\Big|_{t=0},\frac{\partial \Phi}{\partial x_i}\bigg\rangle
=2f\frac{\partial f}{\partial x_i}
$$
Therefore
\begin{equation}
\label{N2t0}
\frac{\partial^2N}{\partial t^2}\Big|_{t=0}=\bigg\langle\frac{\partial^2N}{\partial t^2}\Big|_{t=0},\Phi\bigg\rangle\Phi
+\sum_{k,l=1}^ng^{kl}\bigg\langle\frac{\partial^2N}{\partial t^2}\Big|_{t=0},\frac{\partial \Phi}{\partial x_k}\bigg\rangle\frac{\partial \Phi}{\partial x_l}
\end{equation}
$$
=-|\mathrm{grad}f|^2\Phi+2f\mathrm{grad}f
$$

By (\ref{hijt})) we have
\begin{equation}
\label{hij2t}
\frac{\partial^2 \mathbf{h}_{ij}}{\partial t^2}=
-\bigg\langle \frac{\partial^4 X}{\partial t^2\partial x_i\partial x_j} ,N \bigg\rangle
-2\bigg\langle \frac{\partial^3 X}{\partial t\partial x_i\partial x_j} ,\frac{\partial N}{\partial t} \bigg\rangle
-\bigg\langle \frac{\partial^2 X}{\partial x_i\partial x_j} ,\frac{\partial^2 N}{\partial t^2}\bigg\rangle
\end{equation}
Then by (\ref{Nt0}) and (\ref{N2t0})
\begin{equation}
\label{hij2t0}
\frac{\partial^2 \mathbf{h}_{ij}}{\partial t^2}\Big|_{t=0}
=4\,\frac{\partial f}{\partial x_i}\frac{\partial f}{\partial x_j}+\left(\varphi''(0)-|\mathrm{grad}f|^2\right)g_{ij}
\end{equation}

Finally, by (\ref{Ht})
\begin{equation}
\label{H2t}
\frac{\partial^2 H}{\partial t^2}=\sum_{i,j=1}^n\frac{\partial^2 \mathbf{g}^{ij}}{\partial t^2}\mathbf{h}_{ij}
+2\sum_{i,j=1}^n\frac{\partial \mathbf{g}^{ij}}{\partial t}\frac{\partial \mathbf{h}_{ij}}{\partial t}
+\sum_{i,j=1}^n\mathbf{g}^{ij}\frac{\partial^2 \mathbf{h}_{ij}}{\partial t^2}
\end{equation}
Then by (\ref{ijgt0}), (\ref{ijg2t0}), (\ref{hijt0}) and (\ref{hij2t0})
$$
\frac{\partial^2 H}{\partial t^2}\Big|_{t=0}=
2f^2-\varphi''(0)-\frac{(n-2)}{n}|\mathrm{grad}f|^2
+\frac{4}{n}f\Delta f
$$
\end{proof}

\begin{lemma}	
\label{lemma4}
In the same hypothesis as lemma\ref{lemma2} we have
\begin{equation}
\label{sqrtG3t0}			
\frac{\partial^3}{\partial t^3}\sqrt{\mathbf{g}}\, \Big|_{t=0}=
\bigg((n^3-3n^2+2n)f^3+n\varphi'''(0)\bigg)\sqrt{g}
\end{equation}
$$
+\bigg((3n-6)\,f|\mathrm{grad}f|^2
+(3n^2-3n)\varphi''(0)f\bigg)\sqrt{g}
$$
\begin{equation}
\label{H3t0}		
\frac{\partial^3\mathbf{H}}{\partial t^3}\Big|_{t=0}=
-n\varphi'''(0)-6nf^3
+(9n-18)f|\mathrm{grad}f|^2
\end{equation}
$$
+6n\varphi''(0)f
-18f^2\Delta f+6\varphi''(0)\Delta f
+\frac{(3n+6)}{n}\,|\mathrm{grad}f|^2\Delta f
$$
\end{lemma}
\begin{proof}
Deriving the expression in (\ref{gij2t}) we have
\begin{equation}
\label{gij3t}
\frac{\partial^3\mathbf{g}_{ij}}{\partial t^3}=
6(f+\varphi'(t) )\varphi''(t) g_{ij}+2(1+tf+\varphi(t) )\varphi'''(t) g_{ij}
\end{equation}
For $t=0$ we have
\begin{equation}
\label{gij3t0}
\frac{\partial^3\mathbf{g}_{ij}}{\partial t^3}\Big|_{t=0}
=\Big(6f\varphi''(0)+2\varphi'''(0)\Big) g_{ij}
\end{equation}

Deriving the expression in (\ref{ijg2t}) we have
\begin{equation}
\label{ijg3t}
\frac{\partial^3\mathbf{g}^{ij}}{\partial t^3}=
-2\sum_{r,p=1}^{n}\left( 3(f+\varphi'(t))\varphi''(t)+
(1+tf+\varphi(t) )\varphi'''(t) \right)g_{pr}\mathbf{g}^{ip}\mathbf{g}^{rj}
\end{equation}
$$
-4\sum_{r,p=1}^{n}\left( \frac{\partial f}{\partial x_p}\frac{\partial f}{\partial x_r}
+(f+\varphi'(t))^2g_{pr}+(1+tf+\varphi(t) )\varphi''(t) g_{pr}\right)\frac{\partial\mathbf{g}^{ip}}{\partial t}\mathbf{g}^{rj}
$$
$$
-4\sum_{r,p=1}^{n}\left( \frac{\partial f}{\partial x_p}\frac{\partial f}{\partial x_r}
+(f+\varphi'(t))^2g_{pr}+(1+tf+\varphi(t) )\varphi''(t) g_{pr}\right)\frac{\partial\mathbf{g}^{rj}}{\partial t}\mathbf{g}^{ip}
$$
$$
-2\sum_{r,p=1}^{n}\left( t\frac{\partial f}{\partial x_p}\frac{\partial f}{\partial x_r}
+(1+tf+\varphi(t) )(f+\varphi'(t) )g_{pr}\right)\frac{\partial^2\mathbf{g}^{ip}}{\partial t^2}\mathbf{g}^{rj}
$$
$$
-4\sum_{r,p=1}^{n}\left( t\frac{\partial f}{\partial x_p}\frac{\partial f}{\partial x_r}
+(1+tf+\varphi(t) )(f+\varphi'(t) )g_{pr}\right)\frac{\partial\mathbf{g}^{ip}}{\partial t}
\frac{\partial\mathbf{g}^{rj}}{\partial t}
$$
$$
-2\sum_{r,p=1}^{n}\left( t\frac{\partial f}{\partial x_p}\frac{\partial f}{\partial x_r}
+(1+tf+\varphi(t) )(f+\varphi'(t) )g_{pr}\right)\frac{\partial^2\mathbf{g}^{rj}}{\partial t^2}\mathbf{g}^{ip}
$$
Then for $t=0$, by (\ref{ijgt0}) and (\ref{ijg2t0})
\begin{equation}
\label{ijg3t0}
\frac{\partial^3\mathbf{g}^{ij}}{\partial t^3}\Big|_{t=0}=
\left(\frac{24}{n}f|\mathrm{grad}f|^2-24f^3
\right)g^{ij}
\end{equation}
$$
+\Big(18\varphi''(0)f
-2\varphi'''(0)\Big)g^{ij}
$$

Deriving the expression in (\ref{sqrtG2t}) we have
\begin{equation}
\label{sqrtG3t}
\frac{\partial^3}{\partial t^3}\sqrt{\mathbf{g}}
=\frac{1}{2}\bigg(\sum_{i,j=1}^{n}\frac{\partial^2 \mathbf{g}^{ij}}{\partial t^2}\;
\frac{\partial\mathbf{g}_{ij}}{\partial t}\bigg)\sqrt{\mathbf{g}}
+\bigg(\sum_{i,j=1}^{n}\frac{\partial \mathbf{g}^{ij}}{\partial t}\;
\frac{\partial^2\mathbf{g}_{ij}}{\partial t^2}\bigg)\sqrt{\mathbf{g}}
\end{equation}
$$
+\bigg(\sum_{i,j=1}^{n}\frac{\partial \mathbf{g}^{ij}}{\partial t}\;
\frac{\partial\mathbf{g}_{ij}}{\partial t}\bigg)\frac{\partial\sqrt{\mathbf{g}}}{\partial t}	
+\frac{1}{2}\bigg(\sum_{i,j=1}^{n}\mathbf{g}^{ij}\;
\frac{\partial^3\mathbf{g}_{ij}}{\partial t^3}\bigg)\sqrt{\mathbf{g}}
$$
$$
+\bigg(\sum_{i,j=1}^{n}\mathbf{g}^{ij}\frac{\partial^2 \mathbf{g}_{ij}}{\partial t^2}\bigg)\frac{\partial\sqrt{\mathbf{g}}}{\partial t}
+\frac{1}{2}\bigg(\sum_{i,j=1}^{n}\mathbf{g}^{ij}\frac{\partial\mathbf{g}_{ij}}{\partial t}\bigg)\frac{\partial^2\sqrt{\mathbf{g}}}{\partial t^2}
$$
Then by (\ref{gijt0}), (\ref{gij2t0}), (\ref{gij3t0}), (\ref{ijgt0}),
(\ref{ijg2t0}), (\ref{sqrtGt0}) and (\ref{sqrtG2t0})
$$
\frac{\partial^3}{\partial t^3}\sqrt{\mathbf{g}}\, \Big|_{t=0}=
\bigg((n^3-3n^2+2n)f^3+n\varphi'''(0)\bigg)\sqrt{g}
$$
$$
+\bigg((3n-6)\,f|\mathrm{grad}f|^2
+(3n^2-3n)\varphi''(0)f\bigg)\sqrt{g}
$$

We have
\begin{equation}
\label{N3tX}
0=\frac{\partial^3}{\partial t^3}\left\langle N,N\right\rangle
=2\bigg\langle \frac{\partial^3 N}{\partial t^3},N\bigg\rangle
+6\bigg\langle \frac{\partial^2 N}{\partial t^2},\frac{\partial N}{\partial t}\bigg\rangle
\end{equation}
Then by (\ref{Nt0}) and (\ref{N2t0})
$$
\bigg\langle \frac{\partial^3 N}{\partial t^3}\Big|_{t=0},\Phi\bigg\rangle
=-3\,\bigg\langle \frac{\partial^2 N}{\partial t^2}\Big|_{t=0},\frac{\partial N}{\partial t}\Big|_{t=0}\bigg\rangle
=6f|\mathrm{grad}f|^2
$$
Furthermore
\begin{equation}
\label{N3tXxi}
0=\frac{\partial^3}{\partial t^3}\bigg\langle N,\frac{\partial X}{\partial x_i}\bigg\rangle
\end{equation}
$$
=\bigg\langle\frac{\partial^3 N}{\partial t^3},\frac{\partial X}{\partial x_i}\bigg\rangle
+3\,\bigg\langle \frac{\partial^2N}{\partial t^2},\frac{\partial^2X}{\partial t\partial x_i}\bigg\rangle
+3\,\bigg\langle \frac{\partial N}{\partial t},\frac{\partial^3X}{\partial t^2\partial x_i}\bigg\rangle
+\bigg\langle N,\frac{\partial^4X}{\partial t^3\partial x_i}\bigg\rangle
$$
Then by (\ref{Nt0}) and (\ref{N2t0})
$$
\bigg\langle\frac{\partial^3 N}{\partial t^3}\Big|_{t=0},\frac{\partial \Phi}{\partial x_i}\bigg\rangle
=
\left(3|\mathrm{grad}f|^2-6f^2+3\varphi''(0)\right)\frac{\partial f}{\partial x_i}
$$
and
\begin{equation}
\label{N3t0}
\frac{\partial^3 N}{\partial t^3}\Big|_{t=0}=
\bigg\langle\frac{\partial^3 N}{\partial t^3}\Big|_{t=0},\Phi\bigg\rangle\Phi
+\sum_{i,j=1}^ng^{ij}\bigg\langle\frac{\partial^3 N}{\partial t^3}\Big|_{t=0},\frac{\partial \Phi}{\partial x_i}\bigg\rangle \frac{\partial \Phi}{\partial x_j}
\end{equation}
$$
=6f|\mathrm{grad}f|^2\Phi
+\left(3|\mathrm{grad}f|^2-6f^2+3\varphi''(0)\right)\mathrm{grad}f
$$

By (\ref{hij2t}) we have
\begin{equation}
\label{hij3t}
\frac{\partial^3 \mathbf{h}_{ij}}{\partial t^3}=
-\bigg\langle \frac{\partial^5 X}{\partial t^3\partial x_i\partial x_j} ,N \bigg\rangle
-3\bigg\langle \frac{\partial^4 X}{\partial t^2\partial x_i\partial x_j} ,\frac{\partial N}{\partial t} \bigg\rangle
\end{equation}
$$
-3\bigg\langle \frac{\partial^3 X}{\partial t\partial x_i\partial x_j} ,\frac{\partial^2 N}{\partial t^2} \bigg\rangle
-\bigg\langle \frac{\partial^2 X}{\partial x_i\partial x_j} ,\frac{\partial^3 N}{\partial t^3}\bigg\rangle
$$
Then by (\ref{Nt0}), (\ref{N2t0}) and (\ref{N3t0})
\begin{equation}
\label{hij3t0}
\frac{\partial^3 \mathbf{h}_{ij}}{\partial t^3}\Big|_{t=0}=
\left(\varphi'''(0)+3\,f|\mathrm{grad}f|^2\right)g^{ij}
-12f\frac{\partial f}{\partial x_i}\frac{\partial f}{\partial x_j}
\end{equation}
$$
+3|\mathrm{grad}f|^2\left(\frac{\partial^2f}{\partial x_i\partial x_j}
-\sum_{k=1}^n\Gamma_{ij}^k\frac{\partial f}{\partial x_k}\right)
$$

Finally, deriving the expression in (\ref{H2t}) we have
\begin{equation}
\label{H3t}
\frac{\partial^3 H}{\partial t^3}=\sum_{i,j=1}^n\frac{\partial^3 \mathbf{g}^{ij}}{\partial t^3}\mathbf{h}_{ij}
+3\sum_{i,j=1}^n\frac{\partial^2 \mathbf{g}^{ij}}{\partial t^2}\frac{\partial \mathbf{h}_{ij}}{\partial t}
+3\sum_{i,j=1}^n\frac{\partial \mathbf{g}^{ij}}{\partial t}\frac{\partial^2 \mathbf{h}_{ij}}{\partial t^2}
\end{equation}
$$
+\sum_{i,j=1}^n\mathbf{g}^{ij}\frac{\partial^3 \mathbf{h}_{ij}}{\partial t^3}
$$
Then by (\ref{ijgt0}), (\ref{ijg2t0}), (\ref{ijg3t0}), (\ref{hijt0}), (\ref{hij2t0}) and (\ref{hij3t0})
$$
\frac{\partial^3 H}{\partial t^3}\Big|_{t=0}=
-n\varphi'''(0)-6nf^3
+(9n-18)f|\mathrm{grad}f|^2
$$
$$
+6n\varphi''(0)f
-18f^2\Delta f+6\varphi''(0)\Delta f
+\frac{(3n+6)}{n}\,|\mathrm{grad}f|^2\Delta f
$$
\end{proof}

\begin{lemma}
\label{lemma5}	
In the same hypothesis as lemma\ref{lemma2} we have
\begin{equation}
\label{sqrtG4t0}			
\frac{\partial^4}{\partial t^4}\sqrt{\mathbf{g}}\Big|_{t=0}=
(6n^2-30n+32)f^2|\mathrm{grad}f|^2\sqrt{g}
\end{equation}
$$
+(6n^3-18n^2+8n)\varphi''(0)f^2\sqrt{g}
+(n^4-6n^3+11n^2-12n)f^4\sqrt{g}
$$
$$
+(4n^2-4n)\varphi'''(0)f\sqrt{g}
+(3n^2-n)\varphi''(0)^2\sqrt{g}
+n\varphi^{(4)}(0)\sqrt{g}
$$
$$
+\frac{(3n-4)}{n}\,|\mathrm{grad}f|^4\sqrt{g}
+(6n-8)\varphi''(0)|\mathrm{grad}f|^2\sqrt{g}
$$
and
\begin{equation}
\label{H4t0}		
\frac{\partial^4 \mathbf{H}}{\partial t^4}\Big|_{t=0}=-n\varphi^{(4)}(0)
+6n\varphi''(0)^2+8n\varphi'''(0)f+24nf^4
\end{equation}
$$
-36n\varphi''(0)f^2
+\frac{(9n^2-12n-24)}{n}|\mathrm{grad}f|^4
-72\varphi''(0)f\Delta f
$$
$$
-\frac{(48n+96)}{n}f\,|\mathrm{grad}f|^2\Delta f
+96f^3\Delta f
+8\varphi'''(0)\Delta f
$$
$$
+(-72n+48)f^2|\mathrm{grad}f|^2
+(18n-36)\varphi''(0)|\mathrm{grad}f|^2
$$
\end{lemma}
\begin{proof}
Deriving the expression in (\ref{gij3t}) and making $t = 0$ we have we have
\begin{equation}
\label{gij4t0}
\frac{\partial^4\mathbf{g}_{ij}}{\partial t^4}\Big|_{t=0}=6\varphi''(0)^2 g_{ij}
+8f\varphi'''(0) g_{ij}
+2\varphi^{(4)}(0) g_{ij}
\end{equation}

Deriving the expression in (\ref{ijg3t}) and making $t = 0$ we have
\begin{equation}
\label{ijg4t0}
\frac{\partial^4\mathbf{g}^{ij}}{\partial t^4}\Big|_{t=0}=
\left(18\varphi''(0)^2-2\varphi^{(4)}(0) \right)g^{ij}
\end{equation}
$$
+\bigg(120f^4-144\varphi''(0)f^2
+\frac{24}{n^2}|\mathrm{grad}f|^4\bigg)g^{ij}
$$
$$
+\bigg(-\frac{240}{n}f^2|\mathrm{grad}f|^2
+\frac{48}{n}\varphi''(0)|\mathrm{grad}f|^2+24\varphi'''(0)f
\bigg)g^{ij}
$$

Deriving the expression in(\ref{sqrtG3t}) and making $t = 0$ we obtain (\ref{sqrtG4t0}).

To establish (\ref{H4t0}) we first derive the expression in (\ref{N3tX}) and make 
$t = 0$ to obtain
$$
\bigg\langle \frac{\partial^4 N}{\partial t^4}\Big|_{t=0},\Phi\bigg\rangle
=\left(12|\mathrm{grad}f|^2-36f^2+12\varphi''(0)\right)|\mathrm{grad}f|^2
-3|\mathrm{grad}f|^4
$$
Furthermore, deriving the expression in(\ref{N3tXxi}) and making $t = 0$ we have
$$
\bigg\langle\frac{\partial^4 N}{\partial t^4},\frac{\partial X}{\partial x_i}\bigg\rangle
=\Big(4\,\varphi'''(0)-24\varphi''(0)f\Big)\frac{\partial f}{\partial x_i}
$$
$$
+\Big(24f^3-36f|\mathrm{grad}f|^2\Big)\frac{\partial f}{\partial x_i}
$$
Then
$$
\frac{\partial^4 N}{\partial t^4}\Big|_{t=0}=
\Big(9|\mathrm{grad}f|^4-36f^2|\mathrm{grad}f|^2
+12\varphi''(0)|\mathrm{grad}f|^2\Big)\Phi
$$
$$
+\Big(24f^3-36f|\mathrm{grad}f|^2
-24\varphi''(0)f
+4\varphi'''(0)\Big)\mathrm{grad}f
$$

Deriving the expression in (\ref{hij3t}) and making $t = 0$ we have
$$
\frac{\partial^4 \mathbf{h}_{ij}}{\partial t^4}\Big|_{t=0}=
\Big(9|\mathrm{grad}f|^4-12f^2|\mathrm{grad}f|^2
+6\varphi''(0)|\mathrm{grad}f|^2+\varphi^{(4)}(0)\Big)g_{ij}
$$
$$
+\Big(48f^2-24|\mathrm{grad}f|^2-24\varphi''(0)\Big)\frac{\partial f}{\partial x_i}\frac{\partial f}{\partial x_j}
$$
$$
-24f|\mathrm{grad}f|^2\left(\frac{\partial^2f}{\partial x_i\partial x_j}
-\sum_{k=1}^n\Gamma_{ij}^k\frac{\partial f}{\partial x_k}\right)
$$

Deriving the expression in (\ref{H3t}) we have
$$
\frac{\partial^4 H}{\partial t^4}=
\frac{1}{n}\sum_{i,j=1}^n\frac{\partial^4 \mathbf{g}^{ij}}{\partial t^4}\mathbf{h}_{ij}
+\frac{4}{n}\sum_{i,j=1}^n\frac{\partial^3 \mathbf{g}^{ij}}{\partial t^3}\frac{\partial \mathbf{h}_{ij}}{\partial t}
$$
$$
+\frac{6}{n}\sum_{i,j=1}^n\frac{\partial^2 \mathbf{g}^{ij}}{\partial t^2}\frac{\partial^2 \mathbf{h}_{ij}}{\partial t^2}
+\frac{4}{n}\sum_{i,j=1}^n\frac{\partial \mathbf{g}^{ij}}{\partial t}\frac{\partial^3 \mathbf{h}_{ij}}{\partial t^3}
+\frac{1}{n}\sum_{i,j=1}^n\mathbf{g}^{ij}\frac{\partial^4 \mathbf{h}_{ij}}{\partial t^4}
$$
and making $t = 0$ we obtain (\ref{H4t0}).

\end{proof}
\section{Willmore Functional}

We will present the Willmore theorem for sufarce of genus 0 proposed in \cite{L} that given us a characterization for Euclidian sphere in dimension 2. Then, we will prove a local generalization for this result in higher dimensions. 

\begin{theorem}[T. J. Willmore \cite{L}]\label{will}
Let $S$ have genus 0.Then for all $f\in\mathcal{F}$ we have
\begin{equation*}
	W(f)\geq 2 .
\end{equation*}
Moreover, $W(f)=2$ if and only if $f(S)$ is a euclidean sphere.
\end{theorem}

Now we will go work with objects of higher dimension. For this, we will use the volume-preserving variation theory for hypersurface presented in the previous section \ref{Var}. We will begin by defining a generalization for the Willmore functional.

\begin{definition}

For each $t\in(-\varepsilon,\varepsilon)$ we associate a immersion $X_t$ definite in (\ref{X_t}) and hence a embedded hypersurface $M^n_t$ at $\mathbb{R}^{n+1}$ with the same volume of $\mathbb{S}^n,$ such that $M_0=\mathbb{S}^n.$ So, we let's define the functional $\mathcal{W}:(-\varepsilon,\varepsilon) \to  \mathbb{R}$ given by
$$
\mathcal{W}(t)=\frac{1}{n\, vol(\mathbb{S}^n)} \int_{M_t} \mathbf{H}^2\,\mathrm{dM_t}.
$$
where $\mathbf{H}=\mathbf{H}(t)$ is the mean curvature non-normalized associated the inherited metric $\mathbf{g}_{ij}=\mathbf{g}_{ij}(t)$ at $M_t=X_t(\mathbb{S}^n)$.
\end{definition}

\begin{theorem}
\label{WillG}
For each $t\in(-\varepsilon,\varepsilon)$ we associate a embedded hypersurface $M_t$ with the same volume of $M_0=\mathbb{S}^n,$ given for the variation $X_t$ in (\ref{X_t}),  with $f$ having zero average. So there is $\delta\in(0,\varepsilon)$ such that for all $t\in(-\delta,\delta)$ we have
\begin{equation}
\mathcal{W}(t)\geq n.
\end{equation}
Moreover, $\mathcal{W}(t)=n$ for each $t\in(-\delta,\delta)$ if and only if $M_t=\mathbb{S}^n$.
\end{theorem}

\begin{proof}

We will prove that the $\mathcal{W}$ functional attains its local minimum in $t=0.$

\allowdisplaybreaks{\begin{eqnarray*}
		\mathcal{W}'(0)&=&\frac{1}{n\,vol(\mathbb{S}^n)}\frac{\partial}{\partial t}\Big|_{t=0}\left(\int_{M_t}\mathbf{H}^2\, \mathrm{dM_t}\right) \\
		&=&\frac{1}{n\,vol(\mathbb{S}^n)}\left(\int_U\frac{\partial \mathbf{H}^2}{\partial t}\Big|_{t=0} \sqrt{g}\, \mathrm{d}x+
		\int_U\frac{\partial \sqrt{\mathbf{g}}}{\partial t}\Big|_{t=0}\mathrm{d}x\right)\\
		&=&\frac{1}{n\,vol(\mathbb{S}^n)}\left((n^3-2n^2)\int_{\mathbb{S}^n}f\,\mathrm{d}\mathbb{S}^n-2n\int_{\mathbb{S}^n}\Delta f\,\mathrm{d}\mathbb{S}^n\right)\\
		&=&0,
\end{eqnarray*}}
Here we use, for $n\geq 3$, that $X_t$ is variation volume-preserving, then by (\ref{varphit0})
$$
\int_{\mathbb{S}^n}f\mathrm{d}\mathbb{S}^n=0.
$$
So the sphere is a critical point for the $\mathcal{W}$ functional. Moreover,  by calculating the second differential, we get
\allowdisplaybreaks{\begin{eqnarray*}
		\mathcal{W}''(0)&=&\frac{1}{n\,vol(\mathbb{S}^n)}\frac{\partial^2}{\partial t^2}\Big|_{t=0}\left(\int_{M_t}\mathbf{H}^2\, \mathrm{dM_t}\right)\\
		&=&\frac{1}{n\,vol(\mathbb{S}^n)}\left(\int_{U}\frac{\partial^2 \mathbf{H}^2}{\partial t^2}\Big|_{t=0}\sqrt{g}\,\mathrm{d}x+2\int_U\frac{\partial \mathbf{H}^2}{\partial t}\Big|_{t=0}\frac{\partial \sqrt{\mathbf{g}}}{\partial t}\Big|_{t=0}\mathrm{d}x\right) \\
		&=&\frac{1}{n\, vol(\mathbb{S}^n)}\bigg((-2n^3+4n^2)\int_{\mathbb{S}^n}f^2+2\int_{\mathbb{S}^n}(\Delta f)^2\\
		&&+(2n^2-6n)\int_{\mathbb{S}^n}|\nabla f|^2\bigg)\\
\end{eqnarray*}}
We want to determine the sign of $\mathcal{W}''(0)$. For this we will consider an orthogonal base $\{\phi_i\}_{i\in\mathbb{N}}$ of the space $L^2(\mathbb{S}^n)$, formed by the eigenfunction of the Laplace operator on $\mathbb{S}^n$, $-\Delta \phi_i=\beta_i\phi_i,$ $\beta_0=0$, $\beta_1=\dots =\beta_{n+1}=n$, $\phi_0$ is constant and 
$\phi_1, \dots ,\phi_{n+1}$ are known as the first harmonic spheres. By (\ref{varphit0}) the function $f$ is orthogonal to a constant. So, 
$$
f=\sum_{i=1}^{\infty}a_i\phi_i
$$
and
\begin{equation*}
\mathcal{W}''(0)=\frac{2}{n\, vol(\mathbb{S}^n)}\sum_{i=1}^{\infty}(\beta_i-n)(\beta_i+n^2-2n)\,a^2_i\!\int_{\mathbb{S}^n}\phi_i^2.
\end{equation*}
Then $\mathcal{W}''(0)\geq 0$, with equality if and only if $a_i=0$ for all $i\geq n+2$. 

We then have to study the case where
$$
f=\sum_{i=1}^{n+1}a_i\phi_i
$$
i.e., $\Delta f=-n\,f$. In this case we have the following lemma.
\begin{lemma}
If $\Delta f=-n\,f$ then
\begin{equation}
\label{zero}
\int_{\mathbb{S}^n}f^3=
\int_{\mathbb{S}^n}f|\mathrm{grad}f|^2=0 \, ,
\end{equation}
\begin{equation}
\label{intg4}
\int_{\mathbb{S}^n}f^4=\frac{3(n+1)}{(n+3)|\mathbb{S}^n|}\left(\int_{\mathbb{S}^n}f^2\right)^2\!\!\!,
\end{equation}
\begin{equation}
\label{cor4gradg}
\int_{\mathbb{S}^n}|\mathrm{grad}f|^4=\frac{n(n+1)(n+2)}{(n+3)|\mathbb{S}^n|}\left(\int_{\mathbb{S}^n}f^2\right)^2\!\!\!,
\end{equation}
\begin{equation}
\label{g2grad2g}
\int_{\mathbb{S}^n}f^2|\mathrm{grad}f|^2=
\frac{n(n+1)}{(n+3)|\mathbb{S}^n|}\left(\int_{\mathbb{S}^n}f^2\right)^2\!\!\!,
\end{equation}
\begin{equation}
\label{varphi2txi}
\varphi''(0)=
-\frac{n}{|\mathbb{S}^n|}\int_{\mathbb{S}^n}f^2\!,
\end{equation}
and
\begin{equation}
\label{varphi4txi}
\varphi^{(4)}(0)=
\frac{(4n^3-40n^2+52n-24)}{(n+3)|\mathbb{S}^n|^2}\left(\int_{\mathbb{S}^n}f^2\right)^2\!\!\!.
\end{equation}
\begin{proof}    
Let $\Phi:U\subset\mathbb{R}^n\rightarrow\mathbb{S}^n\subset \mathbb{R}^{n+1}$ 
be a parametrization of the Euclidean sphere such that $\Phi(U)=\mathbb{S}^n-\{p_0\}$, $p_0\in\mathbb{S}^n$.
Note that $Y:U \times (0,\infty )\to \mathbb{R}^{n+1}$ 
given by $Y(x, r)=r\,\Phi(x)$ is a parametrization of  $\mathbb{R}^{n+1}$ known as spherical coordinates. In this coordinate system the Laplace operator has the form
$$
\Delta_{\mathbb{R}^{n+1}}=\frac{\partial^2}{\partial r^2}+\frac{n}{r}\frac{\partial}{\partial r}+\frac{1}{r^2}\Delta_{\mathbb{S}^n}.
$$

Let us denote by $y^i:\mathbb{S}^n\to \mathbb{R}$, $i=1, \dots , n+1$, the coordinate functions of $\mathbb{S}^n$ given by
\begin{equation}
\label{xi}
y^i(\Phi(x))=\Phi_i(x)
\end{equation} 
Such functions are called first spherical harmonics. The coordinate functions
$y^i:\mathbb{R}^{n+1}\to \mathbb{R}$ are harminics in $\mathbb{R}^{n+1}$. On the other hand
$$
y^i(X(x, r))=r\,\Phi_i(x).
$$
So,
$$
0=\Delta_{\mathbb{R}^{n+1}}y^i=\frac{n}{r}\,\Phi_i+\frac{1}{r}\,\Delta_{\mathbb{S}^n}\Phi_i
$$
and by (\ref{xi}) we have
$$
-\Delta_{\mathbb{S}^n}y^i=n\,y^i
$$
In addition, the gradient of the functions $y^i:\mathbb{S}^n\to \mathbb{R}$ satisfy
\begin{equation}
\label{ProdGrad}
\langle\mathrm{grad}\,y^i, \mathrm{grad}\,y^j\rangle =\delta_{ij}-y^iy^j.
\end{equation} 
In fact, to prove this we can assume, without loss of generality, that
$i,j\leq n$ and use the sphere's coordinate system $w:B_1\to \mathbb{S}^n$ given by
$$
w(x)=w(x_1,\dots ,x_n)=\left(x_1,\dots ,x_n, \sqrt{1-|x|^2}\;\right)
$$
In this coordinate system we have $y^i(w(x))=x_i$, for all $i\leq n$ and
$$
\frac{\partial w}{\partial x_i}=e_i-\frac{x_i}{\sqrt{1-|x|^2}}\,e_{n+1},
$$
where $\{e_1,\dots ,e_{n+1}\}$ is the canonical base of $\mathbb{R}^{n+1}$. Then, for $i,j\leq n$,
$$
\mathrm{g}_{ij}(x)=\delta_{ij}+\frac{x_ix_j}{1-|x|^2},\;\;\;\;\mathrm{g}^{ij}(u)=\delta_{ij}-u_iu_j
$$
and
$$
\mathrm{grad}\,y^i(w(x))=\sum_{j,k=1}^n\mathrm{g}^{jk}(x)\frac{\partial (y^i\circ w)}{\partial x_j}\frac{\partial}{\partial x_k}
=\sum_{k=1}^n\mathrm{g}^{ik}(x)\frac{\partial}{\partial x_k}
$$
$$
=\left(\sum_{k=1}^n(\delta_{ik}-x_ix_k)e_k\right)-x_i\sqrt{1-|x|^2}\,e_{n+1}
$$
So, at the point $w(x)$,
$$
\langle\mathrm{grad}\,y^i,\mathrm{grad}\,y^j\rangle
=\sum_{k=1}^n(\delta_{ik}-x_ix_k)(\delta_{jk}-x_jx_k)+x_ix_j(1-|x|^2)
$$
$$
=\delta_{ij}-y^i(w(x))\, y^j(w(x)).
$$

We must remember that the first harmonic spheres form a basis for the self-space associated with the self-value $n$. In addition they are $L^2$-orthogonal two by two. Indeed, by  (\ref{ProdGrad}), 
$$
\int_{\mathbb{S}^n}y^iy^j=-\frac{1}{n}\int_{\mathbb{S}^n}y^i\Delta y^j
=\frac{1}{n}\int_{\mathbb{S}^n}\langle\mathrm{grad}y^i,\mathrm{grad}y^j\rangle
=\frac{1}{n}\int_{\mathbb{S}^n}\bigg(\delta_{ij}-y^iy^j\bigg)
$$
Then
\begin{equation}
\label{intyiyj}
\int_{\mathbb{S}^n}y^iy^j=\frac{|\mathbb{S}^n|}{n+1}\,\delta_{ij}
\end{equation}

We also have to for any $i,j,k=1,\dots ,n+1$ we have
\begin{equation}
\label{yiyjyk}
\int_{\mathbb{S}^n}y^iy^jy^k=0 .
\end{equation}
In fact,
$$
\int_{\mathbb{S}^n}y^iy^jy^k=-\frac{1}{n}\int_{\mathbb{S}^n}y^iy^j\Delta y^k=-\frac{1}{n}\int_{\mathbb{S}^n}\Delta (y^iy^j)y^k
$$
$$
=-\frac{1}{n}\int_{\mathbb{S}^n}\bigg((\Delta y^i)y^jy^k+y^i(\Delta y^j)y^k+2y^k\langle\mathrm{grad}y^i,\mathrm{grad}y^j\rangle\bigg)
$$
$$
=2\int_{\mathbb{S}^n}y^iy^jy^k-\frac{2}{n}\int_{\mathbb{S}^n}y^k\,(\delta_{ij}-
y^iy^j)
=\frac{2n+2}{n}\int_{\mathbb{S}^n}y^iy^jy^k=0
$$
Now suppose that $\Delta f=-nf$, that is,
$$
f=\sum_{i=1}^{n+1}a_i\,y^i.
$$
So for (\ref{yiyjyk})
$$
\int_{\mathbb{S}^n}f^3=\sum_{i=1}^{n+1}a_ia_ja_k\int_{\mathbb{S}^n}y^iy^jy^k=0
$$
In addition we also have by (\ref{ProdGrad})
$$
|\mathrm{grad}\,f|^2=\sum_{i,j=1}^{n+1}a_i\,a_j\,\langle\mathrm{grad}\,y^i,\mathrm{grad}\,y^j\rangle=\sum_{i=1}^{n+1}a_i^2-f^2
$$
that is, by (\ref{intyiyj})
\begin{equation}
\label{grdf2}
|\mathrm{grad}\,f|^2=-f^2+\frac{(n+1)}{|\mathbb{S}^n|}\int_{\mathbb{S}^n}f^2
\end{equation}
Then
$$
\int_{\mathbb{S}^n}f|\mathrm{grad}f|^2=
-\int_{\mathbb{S}^n}f^3=0
$$
Let us now prove (\ref{intg4}). To this end we have that
$$
\int_{\mathbb{S}^n}y^iy^jy^ky^l=-\frac{1}{n}\int_{\mathbb{S}^n}y^iy^jy^k\Delta y^l
=-\frac{1}{n}\int_{\mathbb{S}^n}y^l\Delta(y^iy^jy^k)
$$
$$
=-\frac{1}{n}\int_{\mathbb{S}^n}y^l(y^iy^j\Delta y^k+y^iy^k\Delta y^j+y^jy^k\Delta y^i)
$$
$$
-\frac{2}{n}\int_{\mathbb{S}^n}y^l(y^i\langle \mathrm{grad}\,y^j,\mathrm{grad}\,y^k\rangle +
y^j \langle \mathrm{grad}\,y^i,\mathrm{grad}\,y^k\rangle + y^k \langle \mathrm{grad}\,y^i,\mathrm{grad}\,y^j\rangle )
$$
$$
=\frac{3n+6}{n}\int_{\mathbb{S}^n}y^iy^jy^ky^l-\frac{2\delta_{jk}}{n}\int_{\mathbb{S}^n}y^iy^l
-\frac{2\delta_{ik}}{n}\int_{\mathbb{S}^n}y^jy^l-\frac{2\delta_{ij}}{n}\int_{\mathbb{S}^n}y^ky^l
$$
So, for (\ref{intyiyj}) we have,
\begin{equation}
\label{yiyjykyl}
\int_{\mathbb{S}^n}y^iy^jy^ky^l=\frac{|\mathbb{S}^n|}{(n+1)(n+3)}(\delta_{jk}\delta_{il}+\delta_{ik}\delta_{jl}+\delta_{ij}\delta_{kl})
\end{equation}
So that
$$
\int_{\mathbb{S}^n}f^4=\sum_{i,j,k,l=1}^{n+1}a_ia_ja_ka_l\int_{\mathbb{S}^n}y^iy^jy^ky^l
$$
$$
=\frac{|\mathbb{S}^n|}{(n+1)(n+3)}\sum_{i,j,k,l=1}^{n+1}a_ia_ja_ka_l(\delta_{jk}\delta_{il}+\delta_{ik}\delta_{jl}+\delta_{ij}\delta_{kl})
$$
$$
=\frac{3|\mathbb{S}^n|}{(n+1)(n+3)}\sum_{i,j=1}^{n+1}a_i^2a_j^2
=\frac{3|\mathbb{S}^n|}{(n+1)(n+3)}\left(\sum_{i=1}^{n+1}a_i^2\right)^2
$$
$$
=\frac{3|\mathbb{S}^n|}{(n+1)(n+3)}\left(\frac{(n+1)}{|\mathbb{S}^n|}\int_{\mathbb{S}^n}f^2\right)^2
=\frac{3(n+1)}{(n+3)|\mathbb{S}^n|}\left(\int_{\mathbb{S}^n}f^2\right)^2
$$

Now, to prove (\ref{cor4gradg}), we're going to use (\ref{grdf2}) and get
$$
\int_{\mathbb{S}^n}|\mathrm{grad}f|^4=\int_{\mathbb{S}^n}f^4
+\frac{n^2-1}{|\mathbb{S}^n|}\left(\int_{\mathbb{S}^n}f^2\right)^2
$$
and by (\ref{intg4}) we have (\ref{cor4gradg}).

By (\ref{grdf2}) we have
$$
\int_{\mathbb{S}^n}f^2|\mathrm{grad}f|^2=
-\int_{\mathbb{S}^n}f^4
+\frac{n+1}{|\mathbb{S}^n|}\left(\int_{\mathbb{S}^n}f^2\right)^2
$$
and by (\ref{intg4}) we obtain (\ref{g2grad2g}).

Let us now prove (\ref{varphi2txi}). 
For this purpose, we observe that the variation in (\ref{X_t}) is by hypersurfaces of constant volume. Then by (\ref{sqrtG2t0})
$$
0=\int_{U}\frac{\partial^2\sqrt{\mathbf{g}}}{\partial t^2}\Big|_{t=0}\,dx
=\int_{\mathbb{S}^n}|\mathrm{grad}f|^2+(n^2-n)\int_{\mathbb{S}^n}f^2
+n|\mathbb{S}^n|\varphi''(0)
$$ 
and taking the value of $\varphi''(0)$ in the expression above we get (\ref{varphi2txi}).

Finally, it remains to prove (\ref{varphi4txi}). In a similar way the proof of the previous item, we integrate (\ref{sqrtG4t0}) and we use what we have already proven in this lemma. Then (\ref{varphi4txi}) is true. 
\end{proof}

We can now return to the proof of the theorem \ref{WillG}.
By Lemma (\ref{lemma3}), (\ref{lemma4}) and (\ref{zero})
\allowdisplaybreaks{\begin{eqnarray*}
\mathcal{W}'''(0)&=&
(2n^2-18n+80)\int_{\mathbb{S}^n}f^3 \\
 &-&\frac{(6n^2-30n+38)}{n}\int_{\mathbb{S}^n}f|\mathrm{grad}f|^2=0
\end{eqnarray*}}
So it will be necessary to derive the functional one more time.
\allowdisplaybreaks{\begin{eqnarray*}
\mathcal{W}^{(4)}(0)&=&-2|\mathbb{S}^n|\varphi^{(4)}(0)
+6|\mathbb{S}^n|\varphi''(0)^2 \\
&-&(24n^2-48n)\varphi''(0)\int_{\mathbb{S}^n}f^2 \\
&-&(24n^2-118n+120)\int_{\mathbb{S}^n}f^4 \\
&+&\frac{(12n^2-12n-24)}{n^2}\int_{\mathbb{S}^n}|\mathrm{grad}f|^4 \\
&-&\frac{(12n^3-84n^2+327n-366)}{n}\int_{\mathbb{S}^n}f^2|\mathrm{grad}f|^2
\end{eqnarray*}}
and by the previous Lemma
$$
\mathcal{W}^{(4)}(0)=
\frac{(12n^5+34n^4+161n^3-95n^2-42n-48)}{n(n+3)|\mathbb{S}^n|}\left(\int_{\mathbb{S}^n}f^2\right)^2
$$
Soon $\mathcal{W}^{(4)}(0)\geq 0$ with equality if and only if $f\equiv 0$, and we can conclude the theorem.
\end{lemma} 
\end{proof}

From this result one can to conjecture if very complete hypersurface $M\subset\mathbb{R}^{n + 1}$, with the same volume of the unit sphere, must to satisfy
$W(M)\geq n$, with equality if only if $M$ is the Euclidean sphere. Here
$$
W(M)=\frac{1}{n\, vol(M)}\int_M\mathbf{H}^2\mathrm{d}M
$$
and $\mathbf{H}$ is the mean curvature non-normalised of $M$.

\section{Characterization results for sphere.}

\subsection{A maximum property of $\mathbb{S}^2$}
Throughout this section, consider $S\subset\mathbb{R}^3$ a differentiable surface of class $\mathcal{C}^{\infty}$, oriented, closed, and of the same volume as the sphere $\mathbb{S}^n$.

\begin{theorem}
\label{FirstTheorem}
Consider the operator $\mathcal{L}:H^2(S)\rightarrow L^2(S)$ defined by
\begin{equation*}
	\mathcal{L}=-\Delta-\displaystyle\frac{1}{2}\mathbf{H}^2,
\end{equation*}
where $\Delta$ and $\mathbf{H}$ are, respectively, Laplacian and mean curvature non-normalized of the inherited metric in $S$ from the $\mathbb{R}^3$ metric.
The first eigenvalue $\lambda_1^{0}$ from the $\mathcal{L}$ operator in the $\mathbb{S}^2$ sphere is the global maximum between all the first $\mathcal{L}$ eigenvalues on $S$ genus 0 surface.
\end{theorem}

\begin{proof} Let $u\in H^2(S)$ be the first eigenfunction of $\mathcal{L}$ and $\lambda_1$ the related eigenvalue, that is,
\begin{equation}\label{S2}
-\Delta u-\displaystyle\frac{1}{2}\mathbf{H}^2u=\lambda_1u.
\end{equation}

Since $u$ is a first eigenfunction, we can to consider $u>0$ and by using the expression (\ref{S2}), follows

\begin{equation*}
	\lambda_1=-\displaystyle\frac{\Delta u}{u}-\displaystyle\frac{1}{2}\mathbf{H}^2.
\end{equation*}

Integrating over $f(S)$, where $f:S\rightarrow \mathbb{R}^3$ is an embedding,
\begin{equation*}
	\displaystyle\int_{h(S)}\lambda_1\mathrm{d}S=-\displaystyle \int_{h(S)}\displaystyle\frac{\Delta u}{u}\mathrm{d}S-\displaystyle\frac{1}{2}\displaystyle\int_{h(S)}\mathbf{H}^2\mathrm{d}S.
\end{equation*}
By Green's formule,

\begin{equation*}
	\lambda_1.vol(h(S))=-\displaystyle\int_{h(S)}\displaystyle\frac{|\nabla u|^2}{u^2}\mathrm{d}S-\displaystyle\frac{1}{2}\displaystyle \int_{h(S)}\mathbf{H}^2\mathrm{d}S.
\end{equation*}
But, $vol(h(S))=vol(\mathbb{S}^2)=4\pi$. So,
\begin{equation*}
	\lambda_1=-\displaystyle\frac{1}{4\pi}\displaystyle\int_ {h(S)}\displaystyle\frac{|\nabla u|^2}{u^2}\mathrm{d}S-\displaystyle\frac{1}{8\pi}\displaystyle\int_{h(s)}\mathbf{H}^2\mathrm{d}S.
\end{equation*}
Soon,
\begin{equation*}
	\lambda_1\leq-\displaystyle\frac{1}{8\pi}\displaystyle\int_{h(S)}\mathbf{H}^2\mathrm{d}S.
\end{equation*}
However, by Theorem \ref{will} 
\begin{equation*}
	\lambda_1\leq-\displaystyle\frac{1}{8\pi}\displaystyle\int_{h(S)}\mathbf{H}^2\mathrm{d}S=-\displaystyle\frac{1}{2\pi}\displaystyle\int_{h(S)}\mathrm{H}^2\mathrm{d}S=-W(h)\leq-2=\lambda^{0}_1.
\end{equation*}
Hence, $\lambda_1\leq\lambda^0_1$. And equality occurs if, and only if $\displaystyle\int_{h(S)}\frac{|\nabla u|^2}{u^2}=0\Leftrightarrow|\nabla u|=0\Leftrightarrow\text {u is constant}$. So,

$$
\lambda_1=-\frac{1}{2}\mathbf{H}^2\Leftrightarrow\mathbf{H}\;\text{is constant}.
$$
By Alexandrov's Theorem, we have to $S=\mathbb{S}^2$. This concludes the proof of the result.
\end{proof}

\subsection{A maximum property of $\mathbb{S}^n$}

Consider $M_t^n\subset\mathbb{R}^{n + 1}$, with $t\in(-\varepsilon,\varepsilon)$, a differentiable, oriented, closed hypersurface with the same volume of the sphere $\mathbb{S}^n$

\begin{theorem}
\label{SecondTheorem}
Consider the variation of the $X_t$ sphere given in (\ref{X_t}). For each $t\in(-\varepsilon,\varepsilon)$ we associate a linear differential operator in $H^2(M_t)$

\begin{equation}
\mathcal{L}_t=-\Delta-\displaystyle\frac{1}{n}\mathbf{H}^2,
\end{equation}
where $\Delta$ and $\mathbf{H}$ are, respectively, Laplacian and the mean curvature non-norma-lized related to the metric $\mathbf{g}_{ij}(t)$ induced at $M^n_t=X_t(\mathbb{S}^n)$. Let $u_t\in H^2(X_t(\mathbb{S}^n))$ the first eigenfunction of $\mathcal{L}_t$ and $\lambda_1^t$ the related eigenvalue, that is,
\begin{equation*}\label{TP2}
	\mathcal{L}_t(u_t)=-\Delta u_t-\frac{1}{n}\mathbf{H}^2u_t=\lambda_1^t u_t.
\end{equation*}
So the first eigenvalue $\lambda_1^t$ from the $\mathcal{L} _t$ operator attain a local maximum at $t=0$, that is, the first eigenvalue $\lambda^0_1$ from the $\mathcal{L}_0$ operator defined in $X_0(\mathbb{S}^n)=\mathbb{S}^n$ is maximum among the first eigenvalues of the hypersurfaces $X_t(\mathbb{S}^n)$, $t\in(-\varepsilon,\varepsilon)$.

\end{theorem}

\begin{proof}
Let $u_t\in H^2(X_t(\mathbb{S}^n))$ be the first eigenfunction of $\mathcal{L}_t$ and $\lambda_1^t$ the related eigenvalue. Since $u_t$ is a first eigenfunction, we can always consider it positive. So,

$$
-\Delta u_t-\frac{1}{n}\mathbf{H}^2u_t=\lambda_1^tu_t\;\;\;\Rightarrow\;\;\;\lambda_1^t=-\frac{\Delta u_t}{u_t}-\frac{1}{n}\mathbf{H}^2.
$$
Integrating at $M_t$ the above expression, we get
$$\lambda_1^tvol(M_t)=-\int_{M_t}\frac{\Delta u_t}{u_t}\,\mathrm{dM_t}-\frac{1}{n}\int_{M_t}\mathbf{H}^2\,\mathrm{dM_t}.
$$
Using the Green's formule and the $X_t$ variation volume-preserving of the sphere, it follows that
$$
\lambda_1^t=-\frac{1}{vol(\mathbb{S}^n)}\int_{M_t}\frac{|\mathrm{grad}u_t|^2}{u^2_t}\,\mathrm{dM_t}-\frac{1}{n\,vol(\mathbb{S}^n)}\int_{M_t}\mathbf{H}^2\,\mathrm{dM_t}.
$$
Since the first term of the above expression is always non-positive, we have to

\begin{equation*}\label{TP1}
	\lambda_1^t\leq-\frac{1}{n\,vol(\mathbb{S}^n)}\int_{M_t}\mathbf{H}^2\,\mathrm{dM_t}.
\end{equation*}
Also, by the Theorem \ref{WillG}
\begin{equation*}
	\lambda_1^t\leq-\displaystyle\frac{1}{nvol(\mathbb{S}^n)}\displaystyle\int_{M_t}\mathbf{H}^2\mathrm{dM_t}\leq\mathcal{W}(0)=-\frac{1}{nvol(\mathbb{S}^n)}\int_{\mathbb{S}^n}n^2\mathrm{d}\mathbb{S}^n=-n=\lambda^0_1.	
\end{equation*}
Therefore, $\lambda_1^t\leq\lambda^0_1$, for all $t\in(-\varepsilon,\varepsilon)$. Finally, $\lambda_1^t=\lambda^0_1$ if only if $M_t$ is the unit sphere.
In fact, if $\lambda^t_1=\lambda^0_1$ we have in (\ref{TP1}) the equality, that is,
$$
\int_{M_t}\displaystyle\frac{|\nabla^tu_t|^2}{u^2_t}\mathrm{dM_t}=0.
$$	
Soon, $u_t=constant$ and by (\ref{TP2}), we get
\begin{equation*}
	-\frac{1}{n}\mathbf{H}^2(t)(u_t)=\lambda_1^t(u_t)\;\;\;\Rightarrow\;\;\; \mathbf{H}(t)=\text{constant}.
\end{equation*}
Therefore, by Alexandrov's theorem it follows that $M_t=\mathbb{S}^n$. Consequently, we have that the first eigenvalue in the sphere is a local maximum. This completes the proof of the theorem.

\end{proof}

\bibliographystyle{amsplain}

\end{document}